\documentclass[numbook,envcountsame,smallcondensed,draft]{amsart}

\usepackage{amssymb,amsmath,amsfonts,cite,color}

\DeclareMathOperator{\var}{var}

\newtheorem{theorem}{Theorem}

\newtheorem{proposition}{Proposition}

\newtheorem{lemma}{Lemma}

\theoremstyle{definition}

\newtheorem{question}{Question}

\makeatletter

\renewcommand*\subjclass[2][2010]{\def\@subjclass{#2}\@ifundefined{subjclassname@#1}{\ClassWarning{\@classname}{Unknown edition (#1) of Mathematics Subject Classification; using '2010'.}}{\@xp\let\@xp\subjclassname\csname subjclassname@#1\endcsname}}

\renewcommand{\subjclassname}{\textup{2010} Mathematics Subject Classification}

\makeatother

\begin{document}

\title[Standard elements of the lattice of monoid varieties]{Standard elements of the lattice\\ of monoid varieties}

\thanks{The work is supported by the Ministry of Science and Higher Education of the Russian Federation (project FEUZ-2020-0016)}

\author{S.V.~Gusev}

\address{Ural Federal University, Institute of Natural Sciences and Mathematics, Lenina 51, 620000 Ekaterinburg, Russia}

\email{sergey.gusb@gmail.com}

\date{}

\begin{abstract}
We completely classify all standard elements in the lattice of all monoid varieties. In particular, we prove that an element of this lattice is standard if and only if it is neutral.
\end{abstract}



\maketitle

The article is devoted to investigation of the lattice of varieties of monoids which will be denoted by $\mathbb{MON}$ (when referring to monoid varieties,  we consider monoids as semigroups equipped by an additional \mbox{0-ary} operation that fixes the identity element).  Until recently, this lattice has been studied very little. However, recently, the articles~\cite{Gusev-18-IzVUZ,Gusev-19,Gusev-18-AU,Gusev-Vernikov-18,Jackson-Lee-18} which are devoted to this subject appeared. In particular, the study of the special elements of the lattice $\mathbb{MON}$ has been started in~\cite{Gusev-18-AU}. In this paper, we continue to study them.

Let us recall definitions of special elements which will be used below. An element $x$ of a lattice $L$ is called
\begin{align*}
&\text{\emph{neutral} if}&&\forall\,y,z\in L\colon\ (x\vee y)\wedge(y\vee z)\wedge(z\vee x)\\
&&&\phantom{\forall\,y,z\in L\colon}{}=(x\wedge y)\vee(y\wedge z)\vee(z\wedge x);\\
&\text{\emph{standard} if}&&\forall\,y,z\in L\colon\quad(x\vee y)\wedge z=(x\wedge z)\vee(y\wedge z);\\
&\text{\emph{modular} if}&&\forall\,y,z\in L\colon\quad y\le z\rightarrow(x\vee y)\wedge z=(x\wedge z)\vee y;\\
&\text{\emph{lower-modular} if}&&\forall\,y,z\in L\colon\quad x\le y\rightarrow x\vee(y\wedge z)=y\wedge(x\vee z).
\end{align*}
\emph{Costandard} and \emph{upper-modular} elements are defined dually to standard and lower-modular elements respectively. It is evident that a neutral element is both standard and costandard; a standard element is both modular and lower modular; a costandard element is both modular and upper-modular. Some information about special elements in arbitrary lattices can be found in~\cite[Section~III.2]{Gratzer-11}. 

The neutral and costandard elements of the lattice $\mathbb{MON}$ were completely described in~\cite{Gusev-18-AU}. In this paper, we classify the standard elements of this lattice.

We need some definitions and notation. The free monoid over a countably infinite alphabet is denoted by $F^1$. As usual, elements of $F^1$ are called \emph{words}, while elements of $A$ are said to be \emph{letters}. The words unlike letters are written in bold. Two parts of an identity are connected by the symbol $\approx$, while the symbol $=$ denotes, among other things, the equality relation on the free monoid. The trivial variety of monoids is denoted by $\mathbf T$, while $\mathbf{MON}$ denotes the variety of all monoids.  We denote by $\mathbf{SL}$ the variety of all semilattice monoids. Monoid variety given by an identity system $\Sigma$ is denoted by $\var\Sigma$. Put 
$$
\mathbf C_2=\var\{x^2\approx x^3,\,xy\approx yx\}.
$$ 
For convenience, we formulate the main results of~\cite{Gusev-18-AU}.

\begin{proposition}[{\!\!\cite[Theorem~1.1]{Gusev-18-AU}}]
\label{neutral theorem}
For a monoid variety $\bf V$, the following are equivalent:
\begin{itemize}
\item[\textup{(i)}] $\bf V$ is a modular, lower-modular and upper-modular element of the lattice $\mathbb{MON}$;
\item[\textup{(ii)}] $\bf V$ is a neutral element of the lattice $\mathbb{MON}$;
\item[\textup{(iii)}] $\bf V$ is one of the varieties $\mathbf T$, $\mathbf{SL}$ or $\mathbf{MON}$.\qed
\end{itemize}
\end{proposition}

\begin{proposition}[{\!\!\cite[Theorem~1.2]{Gusev-18-AU}}]
\label{costandard theorem}
For a monoid variety $\bf V$, the following are equivalent:
\begin{itemize}
\item[\textup{(i)}] $\bf V$ is a modular and upper-modular element of the lattice $\mathbb{MON}$;
\item[\textup{(ii)}] $\bf V$ is a costandard element of the lattice $\mathbb{MON}$;
\item[\textup{(iii)}] $\bf V$ is one of the varieties $\mathbf T$, $\mathbf{SL}$, $\mathbf C_2$ or $\mathbf{MON}$.\qed
\end{itemize}
\end{proposition}

The main result of the paper is the following

\begin{theorem}
\label{main result}
For a monoid variety $\mathbf V$, the following are equivalent:
\begin{itemize}
\item[\textup{(i)}] $\mathbf V$ is a modular and lower-modular element of the lattice $\mathbb{MON}$;
\item[\textup{(ii)}] $\mathbf V$ is a standard element of the lattice $\mathbb{MON}$;
\item[\textup{(iii)}] $\mathbf V$ is a neutral element of the lattice $\mathbb{MON}$;
\item[\textup{(iv)}] $\mathbf V$ is one of the varieties $\mathbf T$, $\mathbf{SL}$ or $\mathbf{MON}$.
\end{itemize}
\end{theorem}

We note that the equivalence of the claims (iii) and (iv) of Theorem \ref{main result} follows from Proposition~\ref{neutral theorem}. It is natural to compare Theorem~\ref{main result} and Propositions~\ref{neutral theorem} and~\ref{costandard theorem} with the results concerning the special elements of lattice of semigroup varieties denoted by $\mathbb{SEM}$ (the survey of these results can be found in\cite{Vernikov-15}). The properties of being neutral and standard elements are not equivalent in the lattice $\mathbb{SEM}$ (this fact follows from Theorems~3.3 and 3.4 of~\cite{Vernikov-15}), while the properties of being neutral and costandard elements are equivalent in this lattice (see~\cite[Theorem~3.4]{Vernikov-15}). In contrary with the semigroup case, the properties of being neutral and costandard elements are not equivalent (see Propositions~\ref{neutral theorem} and~\ref{costandard theorem}), while the properties of being neutral and standard elements are equivalent in $\mathbb{MON}$ (Theorem~\ref{main result} of this work). Theorem~\ref{main result} implies that an element of $\mathbb{MON}$ is neutral if and only if it is both modular and lower-modular in $\mathbb{MON}$. However, this is not true for the lattice $\mathbb{SEM}$ (this fact follows from Theorems~3.2 and~3.4 and Corollary~3.9 of~\cite{Vernikov-15}). At the same time, an element of $\mathbb{SEM}$ is neutral if and only if it is both upper-modular and lower-modular (see~\cite[Theorem~3.4]{Vernikov-15}). The question about whether the same result holds in the lattices $\mathbb{MON}$ remains open. Finally, every standard element is  costandard one but a standard element does not have to be costandard one in $\mathbb{MON}$ (Theorem~\ref{main result} and Proposition~\ref{costandard theorem}). At the same time, every costandard element is standard one but the properties of being standard and costandard elements are not equivalent in $\mathbb{SEM}$ (Theorems~3.3 and~3.4 of~\cite{Vernikov-15}).

\smallskip

To prove the main result, we need several auxiliary statements. We start with the fact that is a part of the semigroup folklore (it is noted in~\cite[Section 1.1]{Jackson-Lee-18} and~\cite[Proposition~2.1]{Gusev-18-AU}, for instance).

\begin{proposition}
\label{MON sublattice of SEM}
The map from $\mathbb{MON}$ into $\mathbb{SEM}$ that maps a monoid variety generated by a monoid $M$ to the semigroup variety generated by $M$ is an embedding of the lattice $\mathbb{MON}$ into the lattice $\mathbb{SEM}$.\qed
\end{proposition}

Recall that the variety $\mathbf V$ is said to be \emph{periodic} if all its monoids are periodic and \emph{aperiodic} if it does not contain any non-trivial group. A monoid variety $\mathbf V$ is called \emph{proper} if $\mathbf{V\ne MON}$. The proof of the following statement is similar to the arguments from the second paragraph of Section~2.1 of~\cite{Vernikov-07}.
 
\begin{lemma}
\label{lower-modular is periodic}
Let $\mathbf V$ be a proper monoid variety. If $\mathbf V$ is a lower-modular element of $\mathbb{MON}$ then $\mathbf V$ is periodic.
\end{lemma}

\begin{proof}
Suppose that $\mathbf V$ is non-periodic. Then $\mathbf V$ contains the variety of all
commutative monoids. It is proved~\cite[Lemma 2.16]{Vernikov-08} that the variety of all semigroups is generated by all minimal non-Abelian varieties of groups. This fact and Proposition~\ref{MON sublattice of SEM} imply that there exists a minimal non-Abelian group variety $\mathbf G$ such that $\mathbf G\nsubseteq \mathbf V$. Put $\mathbf W = \mathbf V \vee \mathbf G$. Clearly, $\mathbf V \subset \mathbf W$. As is well known, every semigroup variety that contains the variety of all commutative semigroups is generated by all its nilpotent and so aperiodic members (see~\cite{Volkov-96}, for instance). This fact and Proposition~\ref{MON sublattice of SEM} imply that there exists an aperiodic variety $\mathbf K$ such that $\mathbf K \subseteq \mathbf W$ but $\mathbf K\nsubseteq\mathbf V$. Put $\mathbf Y = \mathbf V \vee \mathbf K$. Clearly, $\mathbf V \subset\mathbf Y \subseteq \mathbf W$. It is proved in~\cite[Lemma~1.4]{Vernikov-07} that if $\mathbf U$ is a semigroup variety and $\mathbf X$ is an aperiodic semigroup variety then every group from the variety $\mathbf U\vee\mathbf X$ belongs to $\mathbf U$. Since $\mathbf G \nsubseteq \mathbf V$, this fact and Proposition~\ref{MON sublattice of SEM} imply that $\mathbf G \nsubseteq \mathbf V \vee \mathbf K = \mathbf Y$. Therefore, the variety $\mathbf G \wedge \mathbf Y$  is commutative, whence $\mathbf G \wedge\mathbf Y \subseteq \mathbf V$. Since $\mathbf V$ is a lower-modular element of $\mathbb{MON}$ and $\mathbf V\subseteq\mathbf Y$, we have
$$
\mathbf V = (\mathbf G \wedge \mathbf Y) \vee \mathbf V = (\mathbf G \vee \mathbf V) \wedge \mathbf Y = \mathbf W \wedge \mathbf Y = \mathbf Y,
$$
a contradiction with $\mathbf V \subset \mathbf Y$. We have proved that the variety $\mathbf V$ is periodic.
\end{proof}

The following notion was introduced by Perkins~\cite{Perkins-69} and often appeared in the literature. For any word $\mathbf w$, let $S(\mathbf w)$ denote the Rees quotient monoid of $F^1$ over the ideal of all words that are not subwords of $\mathbf w$. A word $\mathbf w$ is an \emph{isoterm} for a variety $\mathbf V$ if $\mathbf V$ violates any non-trivial identity of the form $\mathbf w\approx\mathbf w'$.
Put 
\begin{align*}
&\mathbf C_n=\var\{x^n\approx x^{n+1},\,xy\approx yx\} \text{ where } n\ge 2,\\
&\mathbf E=\var\{x^2\approx x^3,\,x^2y\approx xyx,\,x^2y^2\approx y^2x^2\}.
\end{align*}
Note that the variety $\mathbf C_2$ has already introduced before Proposition~\ref{neutral theorem}.

\begin{lemma}
\label{if V is mod element}
Let $\mathbf V$ be a monoid variety that contains the variety $\mathbf E$. Suppose that there exists $n\ge2$ such that $\mathbf V$ does not contain the variety $\mathbf C_{n+1}$. Put
\begin{align*}
\mathbf X=(\mathbf V\vee\mathbf C_{n+1})\wedge\var\{x^ky\approx yx^k\mid k>n\},\\
\mathbf Y=(\mathbf V\vee\mathbf C_{n+1})\wedge\var\{x^ky\approx yx^k\mid k\ge n\}.
\end{align*}
Then $\mathbf Y\subset \mathbf X$ and the varieties $\mathbf X$, $\mathbf Y$ and $\mathbf V$ generate the 5-element non-modular sublattice in $\mathbb{MON}$. In particular, $\mathbf V$ is not a modular element of the lattice $\mathbb{MON}$.
\end{lemma}

\begin{proof}
Evidently, $\mathbf Y\subseteq \mathbf X$. We are going to verify that this inclusion is strict. In view of~\cite[Proposition 4.2]{Gusev-Vernikov-18}, if $\mathbf E$ satisfies an identity $yx^n\approx \mathbf w$ then $\mathbf w=yx^t$ for some $t\ge 2$. If the identity $yx^n\approx \mathbf w$ holds in $\mathbf C_{n+1}$ then it follows from commutative law. Taking into account the inclusion $\mathbf E\subseteq \mathbf V$ we have that the word $yx^n$  is an isoterm for $\mathbf V\vee\mathbf C_{n+1}$. Then $S(yx^n)\in \mathbf V\vee\mathbf C_{n+1}$ by~\cite[Lemma~5.3]{Jackson-Sapir-00}. Evidently, $S(yx^n)$ satisfies the identity $x^ky\approx yx^k$ whenever $k>n$. It follows that $S(yx^n)\in \mathbf X$. On the other hand, $S(yx^n)\notin \mathbf Y$ because $S(yx^n)$ violates the identity
\begin{equation}
\label{x^ny=yx^n}
x^ny\approx yx^n.
\end{equation}
Thus, $\mathbf Y\subset \mathbf X$.

It is well known and can be easily verified that if a monoid variety does not contain $\mathbf C_{n+1}$ then this variety satisfies the identity
\begin{equation}
\label{x^n=x^n+m}
x^n\approx x^{n+m}
\end{equation}
for some natural $m$ (see~\cite[Lemma~2.5]{Gusev-Vernikov-18}, for instance). In particular, an identity of such a form holds in $\mathbf V$. Then $\mathbf V$ violates the identity 
\begin{equation}
\label{x^n+my=yx^n+m}
x^{n+m}y\approx yx^{n+m}.
\end{equation}
Therefore, this identity does not hold in $\mathbf V\vee\mathbf C_{n+1}$, whence $\mathbf X\subset\mathbf V\vee\mathbf C_{n+1}$.

Evidently, $\mathbf V\vee\mathbf X=\mathbf V\vee\mathbf C_{n+1}=\mathbf V\vee\mathbf Y$. To complete the proof it remains to note that $\mathbf V\wedge\mathbf X=\mathbf V\wedge\mathbf Y$. Indeed, the variety $\mathbf V\wedge\mathbf X$ satisfies the identity~\eqref{x^ny=yx^n} because this identity follows from the identities~\eqref{x^n=x^n+m} and~\eqref{x^n+my=yx^n+m}. This implies the required conclusion.
\end{proof}

In fact, the following statement is well known (see~\cite[Lemma~2.1]{Gusev-Vernikov-18}, for instance).
 
\begin{lemma}
\label{group variety}
For a monoid variety $\mathbf V$, the following are equivalent:
\begin{itemize}
\item[\textup{a)}] $\mathbf V$ is a group variety;
\item[\textup{b)}] $\mathbf V$ satisfies an identity ${\bf u}\approx {\bf v}$ such that $\mathbf u$ contains a letter which does not occur in $\mathbf v$;
\item[\textup{c)}] $\mathbf{SL\nsubseteq V}$.\qed
\end{itemize}
\end{lemma}

A variety of monoids is called \emph{completely regular} if it consists of \emph{completely regular monoids}~(i.e., unions of groups).

\begin{proof}[Proof of Theorem~\ref{main result}]
The claims (iii) and (iv) are equivalent by Proposition~\ref{neutral theorem}. The implications (iii) $\Rightarrow$ (ii) $\Rightarrow$ (i) are obvious. It remains to prove the implication  (i) $\Rightarrow$ (iv). Let $\mathbf V$ be a proper monoid variety that is a modular and lower-modular element of the lattice $\mathbb{MON}$. Suppose that $\mathbf V$ is completely regular. If $\mathbf U$ is a completely regular monoid variety that is a modular element in $\mathbb{MON}$ then $\mathbf U$ is commutative because 
$$
\mathbf C_2\vee(\mathbf D\wedge \mathbf U)\subset \mathbf D\wedge(\mathbf C_2\vee \mathbf U),
$$
where $\mathbf D=\var\{x^2\approx x^3,\, x^2y\approx xyx\approx yx^2\}$, otherwise by~\cite[Lemma~3.1]{Gusev-18-AU} and $\mathbf U$ is aperiodic because
$$
\mathbf Q\vee(\mathbf B_{2,3}\wedge \mathbf U)\subset \mathbf B_{2,3}\wedge(\mathbf Q\vee \mathbf U),
$$
where $\mathbf B_{2,3}=\var\{x^2\approx x^3\}$ and $\mathbf Q=\var\{yxyzxz\approx yxzxyxz\}$, otherwise by~\cite[Lemma~3.2]{Gusev-18-AU}. Since every completely regular aperiodic variety is a variety of idempotent monoids, this implies that $\mathbf V\subseteq\mathbf{SL}$. Therefore, $\mathbf V\in\{\mathbf T,\mathbf{SL}\}$.

Suppose now that $\mathbf V$ is a non-completely regular monoid variety. Lemma~\ref{lower-modular is periodic} implies that $\mathbf V$ is periodic. It is well known that $\mathbf V$ satisfies the identity~\eqref{x^n=x^n+m} for some $n\ge2$ and $m\ge1$. The identity~\eqref{x^n=x^n+m} does not hold in $\mathbf C_{n+1}$, whence $\mathbf C_{n+1}\nsubseteq \mathbf V$. Then $\mathbf E\nsubseteq \mathbf V$ by Lemma~\ref{if V is mod element}. Put $\mathbf W=\mathbf V\vee\mathbf E$. Clearly, $\mathbf V\subset\mathbf W$. 

Put $\mathbf{LRB}=\var\{xy\approx xyx\}$. We are going to verify that $\mathbf{LRB}\nsubseteq \mathbf W$. If $\mathbf V$ is non-commutative then Lemmas~2.14 and 4.1 and Proposition~4.2 of~\cite{Gusev-Vernikov-18} imply that $\mathbf V$ satisfies the identity
\begin{equation}
\label{yx^r=x^syx^t}
yx^r\approx x^syx^t
\end{equation} 
for some $s\ge1$, $t\ge0$, $s+t\ge 2$ and $r\ge 2$. If $\mathbf V$ is commutative then $\mathbf V$ satisfies the identity~\eqref{yx^r=x^syx^t} with $s=t=1$ and $r=2$. The variety $\mathbf E$ satisfies the identity 
\begin{equation}
\label{y^2x^r=x^sy^2x^t}
y^2x^r\approx x^sy^2x^t.
\end{equation} 
The identity~\eqref{y^2x^r=x^sy^2x^t} follows from the identity~\eqref{yx^r=x^syx^t}. Therefore, the identity~\eqref{y^2x^r=x^sy^2x^t} holds in $\mathbf W$. On the other hand, $\mathbf{LRB}$ satisfies the identities $y^2x^r\approx yx$ and $x^sy^2x^t\approx xy$. Therefore, $\mathbf{LRB}$ violates the identity~\eqref{y^2x^r=x^sy^2x^t}, whence $\mathbf{LRB}\nsubseteq \mathbf W$.

In view of~\cite[Proposition~4.7]{Wismath-86}, the subvariety lattice of $\mathbf{LRB}$ is the chain $\mathbf T\subset\mathbf{SL}\subset \mathbf{LRB}$. According to Lemma~\ref{group variety}, $\mathbf{SL}\subseteq \mathbf V$ and, therefore, $\mathbf{SL}\subseteq \mathbf W$. It follows that
$$
\mathbf V\vee(\mathbf W\wedge\mathbf{LRB})=\mathbf V\vee\mathbf{SL}=\mathbf V.
$$
On the other hand, in view of~\cite[Corollary~2.6]{Gusev-Vernikov-18}, every non-completely regular monoid variety contains the variety $\mathbf C_2$. It is proved in~\cite[Proposition~4.1]{Lee-12} that $\mathbf E\subset\mathbf C_2\vee\mathbf{LRB}$. This implies that $\mathbf W\subseteq\mathbf V\vee\mathbf{LRB}$. Thus,
$$
\mathbf W\wedge(\mathbf V\vee\mathbf{LRB})=\mathbf W.
$$
Then, since $\mathbf V\subset\mathbf W$, the variety $\mathbf V$ is not a lower-modular element in $\mathbb{MON}$. A contradiction.
\end{proof}

In view of Proposition~\ref{neutral theorem}, an element is neutral in $\mathbb{MON}$ if and only if this element is both modular, lower-modular and upper-modular in $\mathbb{MON}$. Theorem~\ref{main result} establishes more stronger result. Namely, the property of being upper-modular element can be omitted. In view of Propositions~\ref{neutral theorem} and~\ref{costandard theorem}, the variety $\mathbf C_2$ is a costandard (and, therefore, a modular) element  but is not a neutral one in $\mathbb{MON}$. Thus, the property of being lower-modular element cannot be omitted. The following question is still open

\begin{question}
\label{low-modular=neutral?}
Is it true that an arbitrary lower-modular element of the lattice $\mathbb{MON}$ is a neutral element of this lattice?
\end{question}

In conclusion, we note that the properties of being modular and costandard elements are not equivalent in $\mathbb{MON}$. Indeed, Proposition~\ref{costandard theorem} implies that the variety $\mathbf D$ is not a costandard element of $\mathbb{MON}$. At the same time, the following statement is true. 

\begin{proposition}
\label{D is modular}
The variety $\mathbf D$ is a modular element of the lattice $\mathbb{MON}$.
\end{proposition}

\begin{proof}
Suppose that $\mathbf D$ is not a modular element of the lattice $\mathbb{MON}$. Then~\cite[Proposition~2.1]{Jezek-81} implies that there exist varieties $\mathbf U$ and $\mathbf W$ such that $\mathbf U\subset \mathbf W$ and the varieties $\mathbf D$, $\mathbf U$ and $\mathbf W$ generate the 5-element non-modular sublattice in $\mathbb{MON}$. Clearly, $\mathbf D\nsubseteq \mathbf U$ and $\mathbf D\nsubseteq \mathbf W$. It is verified in~\cite[Lemma~2.12]{Gusev-Vernikov-18} that any variety that does not contain $\mathbf D$ is either completely regular or commutative. It follows that $\mathbf W$ is either completely regular or commutative.   

Recall that an element $x$ of a lattice $L$ is called \emph{codistributive} if 
$$
x\wedge(y\vee z)=(x\wedge y)\vee(x\wedge z)
$$ 
for all $y,z\in L$. It is well known that a codistributive element is upper-modular. It is proved in~\cite[Proposition~1.4]{Gusev-18-AU} that each commutative variety of monoids is a codistributive and so an upper-modular element of the lattice $\mathbb{MON}$. Since the varieties $\mathbf D$, $\mathbf U$ and $\mathbf W$ generate the 5-element non-modular sublattice in $\mathbb{MON}$, the variety $\mathbf W$ cannot be commutative. Therefore, $\mathbf W$ is completely regular. Then $\mathbf U$ is completely regular too.

The rest of the proof largely repeats the arguments from the last paragraph of the proof of Theorem~1.2 in~\cite{Gusev-18-AU}. We provide these arguments here for the reader convenience and for the sake of completeness. First, suppose that $\mathbf U$ is a group variety. Then $\mathbf{SL} \nsubseteq \mathbf U$. If $\mathbf W$ consists not only of groups then Lemma~\ref{group variety} implies that $\mathbf{SL} \subseteq \mathbf W$. Then $\mathbf U\wedge \mathbf D=\mathbf T$ but $\mathbf{SL}\subseteq\mathbf W\wedge \mathbf D$. This contradicts $\mathbf U\wedge \mathbf D=\mathbf W\wedge \mathbf D$. So, $\mathbf W$ is a group variety. It is proved in~\cite[Lemma~2.6]{Vernikov-08} that if $\mathbf X$ is an aperiodic semigroup variety and $\mathbf G$ is a group variety then $\mathbf G$ is the largest group subvariety in $\mathbf G\vee\mathbf X$. This fact and Proposition~\ref{MON sublattice of SEM} imply that $\mathbf U$ is the largest group subvariety in $\mathbf U\vee\mathbf D$. But this is impossible because $\mathbf W$ is a group variety and
$$
\mathbf U\subset\mathbf W\subset \mathbf W\vee\mathbf D=\mathbf U\vee\mathbf D.
$$

Suppose now that $\mathbf U$ is not a group variety. Then $\mathbf{SL}\subseteq \mathbf U$ by Lemma~\ref{group variety}. Then $\mathbf{SL}\subseteq \mathbf U\wedge \mathbf D=\mathbf W\wedge \mathbf D\subseteq \mathbf W$, whence $\mathbf W$ is not a group variety. Since $\mathbf U$ is completely regular, $\mathbf U$ satisfies $x\approx x^{n+1}$ for some natural $n$. Let $n$ be the least number with such a property, while $\Sigma$ be an identity basis of the variety $\mathbf U$. We denote by $\zeta$ the endomorphism of the monoid $F^1$ which maps each letter $x$ into the word $x^{n+1}$. Put
$$
\Sigma^\ast=\{\zeta(\mathbf u)\approx \zeta(\mathbf v)\mid \mathbf u\approx\mathbf v\in \Sigma\}.
$$
Obviously, $\mathbf U = \var\{x\approx x^{n+1},\Sigma^\ast\}$. If $\mathbf p \approx \mathbf q\in\Sigma^\ast$ then the words $\mathbf p$ and $\mathbf q$ depend on the same letters by Lemma~\ref{group variety}. It follows from~\cite[Proposition~2.13]{Gusev-Vernikov-18} that if some words $\mathbf s$ and $\mathbf t$ do not contain simple letters (i.e., that occur in the word only once) and depend on the same letters then the identity $\mathbf s\approx \mathbf t$ holds in $\mathbf D$. This fact implies that $\mathbf D$ satisfies $\Sigma^\ast$. Since $\mathbf U\vee \mathbf D=\mathbf W\vee \mathbf D$, the variety $\mathbf W$ satisfies $\Sigma^\ast$ too. Since $\mathbf D$ satisfies $x^2\approx x^3$ and $\mathbf U$ satisfies $x\approx x^{n+1}$, the identity $x^2\approx x^{n+2}$ hokds in $\mathbf U\vee \mathbf D=\mathbf W\vee \mathbf D$. Taking into account that $\mathbf W$ is completely regular, we get that $x\approx x^{n+1}$ holds in $\mathbf W$. Then $\mathbf W\subseteq \mathbf U$. We obtain a contradiction with the choice of the varieties $\mathbf U$ and $\mathbf W$. Thus, we have proved that $\mathbf D$ is a modular element of the lattice $\mathbb{MON}$.
\end{proof}

\subsection*{Acknowledgments.} The author is sincerely grateful to Professor Boris Vernikov for his assistance in the writing of the manuscript.

\end{document}